\documentclass[12pt]{amsart}
\usepackage{amscd,amsmath,amsthm,amssymb}
\usepackage{color}
\usepackage{amsfonts,amssymb,amscd,amsmath,enumerate,verbatim}
\usepackage{lineno}
 \def\NZQ{\mathbb}               
 \def\NN{{\NZQ N}}
 
 \def\ZZ{{\NZQ Z}}

 \def\frk{\mathfrak}               

 \def\mm{{\frk m}}
 
 \def\nn{{\frk n}}
 \def\ab{{\mathbf a}}
 \def\bb{{\mathbf b}}

 \def\opn#1#2{\def#1{\operatorname{#2}}} 
 \opn\chara{char} \opn\length{\ell} \opn\pd{pd} \opn\rk{rk}
 \opn\projdim{proj\,dim} \opn\injdim{inj\,dim} \opn\rank{rank}
 \opn\depth{depth} \opn\grade{grade} \opn\height{height}
 \opn\embdim{emb\,dim} \opn\codim{codim}
 
 \opn\Tr{Tr} \opn\bigrank{big\,rank}
 \opn\superheight{superheight}\opn\lcm{lcm}
 \opn\trdeg{tr\,deg}
 \opn\reg{reg} \opn\lreg{lreg} \opn\ini{in} \opn\lpd{lpd}
 \opn\size{size} \opn\sdepth{sdepth}
 \opn\link{link}\opn\fdepth{fdepth}\opn\lex{lex}
 \opn\div{div} \opn\Div{Div} \opn\cl{cl} \opn\Cl{Cl}
 \opn\Spec{Spec} \opn\Supp{Supp} \opn\supp{supp} \opn\Sing{Sing}
 \opn\Ass{Ass} \opn\Min{Min}\opn\Mon{Mon}
 \opn\Ann{Ann} \opn\Rad{Rad} \opn\Soc{Soc}
 \opn\Im{Im} \opn\Ker{Ker} \opn\Coker{Coker} \opn\Am{Am}
 \opn\Hom{Hom} \opn\Tor{Tor} \opn\Ext{Ext} \opn\End{End}
 \opn\Aut{Aut} \opn\id{id}
 
 \opn\nat{nat}
 \opn\pff{pf}
 \opn\Pf{Pf} \opn\GL{GL} \opn\SL{SL} \opn\mod{mod} \opn\ord{ord}
 \opn\Gin{Gin} \opn\Hilb{Hilb}\opn\sort{sort}
 \opn\aff{aff} \opn
 \con{conv} \opn\relint{relint} \opn\st{st}
 \opn\lk{lk} \opn\cn{cn} \opn\core{core} \opn\vol{vol}  \opn\inp{inp} \opn\nilpot{nilpot}
 \opn\link{link} \opn\star{star}\opn\lex{lex}\opn\set{set}
 \opn\width{wd}
 \opn\gr{gr}
 
 \def\pot#1#2{#1[\kern-0.28ex[#2]\kern-0.28ex]}

 \opn\dirlim{\underrightarrow{\lim}}
 \opn\inivlim{\underleftarrow{\lim}}
 \let\iso=\cong

 \let\to=\rightarrow
 
 \def\Implies{\ifmmode\Longrightarrow \else
         \unskip${}\Longrightarrow{}$\ignorespaces\fi}
 \def\implies{\ifmmode\Rightarrow \else
         \unskip${}\Rightarrow{}$\ignorespaces\fi}
 \def\iff{\ifmmode\Longleftrightarrow \else
         \unskip${}\Longleftrightarrow{}$\ignorespaces\fi}

 \let\:=\colon
 \newtheorem{Theorem}{Theorem}[section]
 \newtheorem{Lemma}[Theorem]{Lemma}
 \newtheorem{Corollary}[Theorem]{Corollary}
 \newtheorem{Proposition}[Theorem]{Proposition}

 \newtheorem{Conjecture}[Theorem]{Conjecture}
 
 \let\epsilon\varepsilon
 \let\kappa=\varkappa
 \def\qed{\ifhmode\textqed\fi
       \ifmmode\ifinner\quad\qedsymbol\else\dispqed\fi\fi}
 \def\textqed{\unskip\nobreak\penalty50
        \hskip2em\hbox{}\nobreak\hfil\qedsymbol
        \parfillskip=0pt \finalhyphendemerits=0}
 \def\dispqed{\rlap{\qquad\qedsymbol}}
 
 \opn\dis{dis}
 \def\pnt{{\raise0.5mm\hbox{\large\bf.}}}
 
 \opn\Lex{Lex}

\begin{document}
\title {On the defining equations of  the tangent cone of a numerical semigroup ring}

 \author {J\"urgen Herzog, Dumitru I.\ Stamate}

\address{J\"urgen Herzog, Fachbereich Mathematik, Universit\"at Duisburg-Essen, Campus Essen, 45117
Essen, Germany} \email{juergen.herzog@uni-essen.de}

\address{Dumitru I. Stamate, Faculty of Mathematics and Computer Science, University of Bucharest, Str. Academiei 14, Bucharest, Romania, and  \newline  \indent
Simion Stoilow Institute of Mathematics of the Romanian Academy, Research group
of the project PN-II-RU-PD-2012-3-0656, P.O.Box 1-764, Bucharest 014700, Romania}
\email{dumitru.stamate@fmi.unibuc.ro}

\dedicatory{Dedicated to  Professor Ernst Kunz on the occasion of his eightieth birthday}

 \begin{abstract}
Let $\ab = a_1 <\dots < a_r$ be a sequence of positive integers, and let $H_{\ab}$ denote the semigroup generated by $a_1, \dots, a_r$.
For an integer $k\geq 0$ we denote by $\mathbf{a}+k$ the shifted sequence $a_1 +k, \dots , a_r +k$.
Fix a field $K$. We show that for all $k \gg 0$ the  tangent cone of the semigroup ring   $K[H_{\ab+k}]$ is
Cohen--Macaulay and   that it has the same Betti numbers  as $K[H_{\ab+k}]$ itself.

As a  consequence, we show that the number of defining equations of the tangent cone of a numerical semigroup ring
is bounded by a value depending only on the width of the semigroup, where the width  of a numerical semigroup is defined to be the difference of
the largest and the smallest  element in the  minimal generating set of the semigroup.
We also provide a conjectured upper bound of the above number of equations and  we verify it in some cases.
\end{abstract}

\thanks{}


\maketitle

\section*{Introduction}

Let $\ab = a_1 < \dots <a_r$ be a sequence of positive integers.
We denote by $\langle a_1, \dots, a_r \rangle$ (or simply by $\langle \ab \rangle$) the subsemigroup of $\NN$ generated by $a_1,\ldots, a_r$.
In other words, $\langle \ab \rangle$  consists of all linear combinations of $a_1,\ldots, a_r$ with non-negative integers.
If $H= \langle a_1, \dots, a_r \rangle$ we call $a_1,\ldots, a_r$ a {\em system of generators} of $H$.
Throughout this paper any subsemigroup $H \subset \NN$ with $0\in H$ is called a {\em numerical semigroup}.
Such a semigroup is finitely generated  and admits  a unique minimal system of generators whose cardinality we denote by $\mu(H)$.
In the literature it is often required as part of the definition of a numerical semigroup that the greatest common divisor of its generators is one.
In the context of this paper it is convenient to drop this requirement.

For any nonnegative integer $k$, we let $\ab+k$ be the shifted sequence $a_1+k, \dots, a_r+k$.
If $H$ is minimally generated by $\ab =  a_1, \dots, a_r$,   we let $H_k= \langle \ab+k \rangle$.
We refer to  $\{H_k \}_{k \in \NN}$ as the {\em shifted family} attached to $H$. Note that even if the $a_i$'s generate $H$ minimally, it may
happen that for some shift $k$ the sequence $\ab +k$  is not a minimal generating set of $H_k$.
Hence in particular,  $(H_k)_\ell$ may be different from  $H_{k+\ell}$.
For example, for $H= \langle 3,5,7 \rangle$  we have  $H_1=\langle 4,6,8 \rangle = \langle 4,6 \rangle$
  and $(H_1)_1= \langle 5, 7 \rangle$. However, $H_2= \langle 5,7,9 \rangle$.
On the other hand,  if  $H= \langle \ab \rangle $ is minimally generated by $\ab= a_1 < \dots < a_r$,
then  for all $k > a_r - 2a_1$, $H_k$ is minimally generated by  the sequence $\ab +k$.

Let $K$ be a field and   $S=K[x_1, \dots, x_r]$ be the polynomial ring over $K$ in the variables $x_1, \dots, x_r$.
Let $\ab = a_1 < \dots <a_r$ be  a sequence of positive integers, and  $\varphi: S \to K[t]$ is the $K$-algebra homomorphism
 with $\varphi(x_i)= t^{a_i}$ for $i=1,\dots, r$, where $K[t]$ is the polynomial ring over $K$ in the variable $t$.
If we let $H= \langle a_1, \dots, a_r \rangle$, then the image of $\varphi$ is the semigroup ring $K[H]$, namely the $K$-subalgebra of $K[t]$ generated by
$t^{a_1}, \dots,t^{a_r}$ over $K$.
We denote the kernel of $\varphi$ by $I(\ab)$.
In the case when $\ab$ is a minimal system of generators of $H$, the ideal $I(\ab)$ only depends on $H$ and we set $I_H= I(\ab)$.

It is known from \cite{He-semi} that the minimal number of  generators $\mu(I_H)$ of $I_H$ is at most $3$  if $r\leq 3$.
On the other hand, even for $r=4$, the number $\mu(I_H)$ may be arbitrarily large, see \cite{Bres}. The more it is surprising that
for any numerical semigroup $H$   there exists an upper bound for the numbers $\mu(I_{H_k})$ independent of $k$, see \cite{Vu}.
This statement was conjectured by H.~Srinivasan and the first author of this paper. It was first proved by
P.~Gimenez, I.~Sengupta and H.~Srinivasan in \cite{GSS}  for numerical semigroups generated by an arithmetic sequence.
This  conjecture and some stronger versions of it have recently been proved in full generality by T. Vu in \cite{Vu}:

\begin{Theorem}{\em (Vu, \cite[Theorem 1.1]{Vu}) }
\label{thm:vu-periodicity}
Let  $\ab = a_1 < \dots <a_r$ be  a sequence of positive integers.
Then the Betti numbers of $I(\ab+k)$ are eventually periodic in $k$ with period $a_r-a_1$.
\end{Theorem}

In this paper we consider the coordinate ring of the tangent cone of $K[H]$, which is nothing but the associated graded ring
$\gr_\mm K[H]$ of $K[H]$ with respect to the maximal ideal $\mm= (t^{a_1}, \dots, t^{a_r})$. Note that $\gr_\mm K[H] \cong S/I_{H}^*$, where
$I_{H}^*$ is the ideal of initial forms of polynomials in $I_H$.  In other words, $I_{H}^*= ( f^* | f\in I_H)$, where for each nonzero $f$,
we let $f^*$ denote  the first nonzero homogeneous component of $f$.

Even though for $r = 3$, as remarked above, $\mu(I_H) \leq 3$,  the number of generators of $I_{H}^{*}$ may be arbitrarily large.
A  family of such examples was first found by T.~Shibuta, see \cite{Goto}.
In Shibuta's family of semigroups the width is unbounded, where by the width of a numerical semigroup $H$, denoted $\width(H)$,
we mean the difference between the largest and the smallest element in the minimal generating set of $H$.
One of the results of this paper (see Corollary \ref{cor:width}) is that there is a global   upper bound for $\mu(I_{H}^*)$ for all numerical semigroups with a given width.
It turns out that this result is a simple consequence of   Vu's Theorem \ref{thm:vu-periodicity} and our following theorem.

\medskip
\noindent
{\bf Theorem ~\ref{thm:main}.}
{\it
Let $H$ be a numerical semigroup. Then there exists $k_0 \in \NN$  such that for all $k\geq k_0$,
the ideal $I_{H_k}$ is minimally generated by a standard basis and such that $\beta_i(I_{H_k})= \beta_i(I^*_{H_k})$ for all $i$.
In particular, $\gr_\mm K[H_k]$ is Cohen--Macaulay for all $k\geq k_0$.
}

\medskip
The methods used to prove that there is a uniform upper bound for $\mu(I_{H}^*)$  for all numerical semigroups with given width do  not provide any explicit bound.
However, there is some computational evidence that  ${\width(H)+1 \choose 2} $ serves as an upper bound,
and indeed this  may be  a sharp  upper bound since it is reached by numerical semigroups generated by integers of suitable intervals.
In Section~\ref{sec:conjectures} we show that this conjectured upper bound is valid for any numerical semigroup $H$ satisfying
the inequality $\mu(I^*_H)\leq \mu(I^*_{\widetilde{H}})$, where $\widetilde{H}$ denotes the  semigroup generated by all integers
in the interval spanned by the smallest and the largest generator of $H$.
In support of our conjecture we show in Proposition~\ref{prop:conj2-arithmetic} that for a numerical semigroup $H$ generated by
an arithmetic sequence one even has $\beta_i(I^*_H)\leq \beta_i(I^*_{\widetilde{H}})$, for all $i$.
In fact such inequalities may be true for any numerical semigroup.
Our results on semigroups generated by an arithmetic sequence depend essentially on the description
of the relations and of the Betti numbers of their semigroup ring as  they are given by Gimenez, Sengupta and Srinivasan \cite{GSS}.

In the final Section \ref{sec:examples} we consider several examples of families of semigroups
in support of our conjectures and describe  for each member $H$ of these classes the ideal $I^*_H$.
The first family is based on a well-known result of J.\ Sally in \cite{Sally}, where she describes the defining ideal of the tangent cone of a local Gorenstein ring satisfying $r=e+d-3$.
Here $r$ is the embedding dimension, $e$ is the multiplicity and $d$ the dimension of the ring.
We call a numerical semigroup a Sally semigroup if the data of its semigroup ring  satisfy this equation.
We show that Sally semigroups exist for any given multiplicity $e\geq 4$.
Another family that we consider is that due to H.~Bresinsky \cite{Bres}.
It is the first known family of $4$-generated numerical semigroups with the property that $\mu(I_H)$ may be arbitrarily large for members $H$ belonging to this family.
We show that the  tangent cone of each  Bresinsky semigroup ring is Cohen-Macaulay (see also  F.\ Arslan \cite{Arslan})
and that the given minimal set of generators of  its defining ideal forms a standard basis.

The other  two families considered in this section are  families of $3$-generated numerical semigroups whose members attain
arbitrarily large width, yet their behavior with respect to  $\mu(I_{H}^*)$ is very different.
For any $a>3$, the ideal $I_{H}^*$ attached to the semigroup $H=\langle a, a+1, 2a+3 \rangle$ is
generated by $\lfloor \frac{a-1}{3} \rfloor +3$ monomials.
For this family the number of generators of $I_{H}^*$ is a quasi-linear function of the width of $H$,
which tends to infinity as  $\width(H)$ tends to infinity. For $a=3b$ we recover the example of T.~Shibuta, treated  with different methods in \cite[Example 5.5]{Goto}.

On the other hand, for any  coprime integers $a,b >3$, we  have $\mu (I_{H}^*) =4$ for all  $H=\langle a, b, ab-a-b \rangle$,
though the width of the semigroups in this family may also be arbitrarily large.


\section{Numerical semigroups of bounded width}
\label{sec:bounded}
For any  nonzero polynomial $f \in S=K[x_1, \dots, x_r]$ we define its initial form $f^*$ as the homogenous component
of $f$ with the least degree and we let $\nu(f) = \deg f^*$,  called the {\em initial degree} of $f$.
For an ideal $I \subset S$ the  ideal $I^* = (f^*| f\in I, f\neq 0)$ is called the {\em initial ideal} of $I$.
Note that $I^*$ is a graded ideal of $S$.

We denote by   $\widehat{S}$  the formal power series ring $K[[x_1, \dots, x_r]]$.  For a nonzero power series $f$,  the homogeneous form  $f^*$ and $\nu(f)$
are defined similarly as for polynomials, and for an  ideal $I \subset \widehat{S}$,  we let, as before,  $I^*\subset S$ be the graded ideal  generated by all $f^*$ with $f\in I$.

Let $I$ be  an ideal in $S$ or in $\widehat{S}$. A set $f_1,\dots,f_m$ of elements of $I$ is called a {\em standard basis} for $I$ if  $I^*= (f_{1}^*,\dots,f_{m}^*)S$.

\medskip
Before giving the proof of Theorem \ref{thm:main} we need a criterion for checking whether a system of generators of
an ideal is a standard basis.

First we make the following observation:
\begin{Lemma}
\label{lemma:one}
Let $I\subset S$ be an ideal. The  polynomials $f_1, \dots, f_m \in I$ form a standard basis of $I$ if and only if they form a standard basis of $I\widehat{S}$.
\end{Lemma}

\begin{proof}
Assume $f_1, \dots, f_m \in I$ form a standard basis of $I$. Let $0 \neq f\in I\widehat{S}$ and set $d=\nu(f)$.  We may write
$f=\sum_{i=1}^{m}{g_i f_i}$ with $g_i \in \widehat{S}$, for $i=1, \dots, m$. Since for the initial part of $f$ only the terms of small degree matter,
we have that
$$
f^* = (\sum_{i=1}^{m}{h_i f_i})^*,
$$
where $h_i$ is the polynomial in $S$ obtained as the sum of the components of $g_i$
of degree at most $d$.  As  the polynomials $f_1, \dots, f_m \in I$ form a standard basis of $I$, we get that
$f^* \in (f_{1}^*, \dots, f_{m}^*)$.

The other implication is straightforward.
\end{proof}

\begin{Lemma}
\label{lemma:crit}
Let $I$ be an ideal of $S=K[x_1, \dots, x_r]$  with $I \subset \nn=(x_1, \dots, x_r)$.
Suppose that $x_1$ is a nonzero divisor on $\widehat{S}/I\widehat{S}$.
Let $\pi: S \to \bar{S}=K[x_2, \dots, x_r]$ be the $K$-algebra homomorphism with $\pi(x_1)=0$ and
$\pi(x_i)= x_i$ for $i>1$, and set $\bar{I} = \pi(I)$.

Let $g_1, \dots, g_m$ be a standard basis of   $\bar{I}$ such that
there exist   polynomials $f_1, \dots, f_m \in I$ with $\pi(f_i)= g_i$ and
$\nu (f_i) = \nu (g_i)$, for $i= 1,\dots, m$. Then
\begin{enumerate}
\item[{\em (a)}]   $f_1, \dots, f_m$ is a standard basis of $I$;
\item[{\em (b)}] $x_1$ is regular on $\gr_\nn (S/I)$;
\item[{\em (c)}] there is an isomorphism
\begin{eqnarray}
\label{eq:iso}
\gr_\nn (S/I) / x_1 \gr_\nn (S/I) \cong \gr_{\bar{\nn}} (\bar{S} / \bar{I}),
\end{eqnarray}
 of graded $K$-algebras, where $\bar{\nn} =\pi(\nn)$.
\end{enumerate}
\end{Lemma}

\begin{proof}
After passing from $S/I$ to the $\nn$-adic completion $\widehat{S}/I\widehat{S}$ we may apply \cite[Theorem 1]{He-reg} and Lemma \ref{lemma:one}.
This proves $(a)$ and $(b)$. Statement $(c)$ follows from $(b)$ and \cite[Lemma, p.~185]{He-reg}.
\end{proof}

Let $H= \langle a_1, \dots, a_r \rangle$ be a numerical semigroup minimally generated by  $a_1, \dots, a_r$.
The proof of  Theorem \ref{thm:main} depends heavily on the following result.
\begin{Theorem}{\em (Vu,  \cite[Corollary 3.7]{Vu})}
\label{thm:vu-equation}
There exists an integer $k_0$ such that for all $k \geq k_0$,   any minimal binomial inhomogeneous generator of
$I_{H_k}$ is of the form  
\begin{eqnarray}
\label{form}
x_1^\alpha u - v x_{r}^\beta,
\end{eqnarray} 
where $\alpha, \beta >0$, and where $u$ and $v$ are monomials in the variables $x_2, \dots, x_{r-1}$  with
$$
\deg x_1^\alpha u  > \deg v x_{r}^\beta.
$$
\end{Theorem}

The main result we wish to prove is the following.

\begin{Theorem}
\label{thm:main}
Let $H$ be a numerical semigroup.
Then there exists $k_0 \in \NN$  such that for all $k\geq k_0$, the ideal $I_{H_k}$ is minimally generated
by a standard basis and such that $\beta_i(I_{H_k})= \beta_i(I_{H_k}^*)$ for all $i$.
In particular, $\gr_\mm K[H_k]$ is Cohen--Macaulay for all $k\geq k_0$.
\end{Theorem}

\begin{proof}
Let $H$ be  minimally generated by $a_1,\dots, a_r$.
We choose $k_0$ as in Theorem~\ref{thm:vu-equation} and larger than $a_r - 2a_1$,
so that $H_k$ is minimally generated by $\ab+k$ for all $k\geq k_0$, and claim that
the minimal set of generators of $I_{H_k}$ as described   in Theorem~\ref{thm:vu-equation} forms a standard basis.

Assume that $I_{H_k}$ is minimally generated by the homogeneous polynomials $f_1, \dots, f_t$ and the polynomials $g_1, \dots, g_s$ where
each $g_i$ is of the form  \eqref{form}.
With  the  notation as in Lemma \ref{lemma:crit}, we have that
$\bar{I}_{H_k} = (\bar{f_1}, \dots, \bar{f_t}, \bar{g_1}, \dots, \bar{g_s})$ and all $\bar{f_i}$ and $\bar{g_j}$ are
homogeneous polynomials, hence  they form a standard basis of  $\bar{I}_{H_k}$.

Since $\nu(f_i)= \nu(\bar{f_i})$ and $\nu(g_j)= \nu(\bar{g_j})$ for all $i, j$,  and   since  $x_1$ is a regular element on
 $K[[H_k]]=\widehat{S}/I_{H_k}\widehat{S}$, we may apply  Lemma \ref{lemma:crit} and conclude that
$f_1,\dots, f_t,g_1, \dots, g_s$ is a standard basis of $I_{H_k}$.

From Lemma \ref{lemma:crit} we also have that  $x_1$ is a form of degree $1$ which is a  regular element
on $\gr_\nn(S/I_{H_k})=S/I_{H_k}^*$.

We have the following chain of equalities
\begin{eqnarray*}
 \beta_i(S/I_{H_k})
& = &\beta_i (\bar{S} / \bar{I}_{H_k})
= \beta_i (\gr_{\bar{\nn}} (\bar{S}/ \bar{I}_{H_k} ))
= \beta_i (\gr_\nn (S/I_{H_k})/ x_1 \gr_\nn (S/I_{H_k})) \\
&= &\beta_i (\gr_\nn (S/I_{H_k}))
= \beta_i (S/I_{H_k}^*).
\end{eqnarray*}

The first equality  holds because $x_1$ is a nonzero divisor on $K[H_k]$, the second equality holds because $\bar{I}_{H_k}$ is a
 homogeneous ideal, the third  because of Lemma \ref{lemma:crit}.  Next, equation four holds because
$x_1$ is a nonzero divisor on $\gr_\nn (S/I_{H_k})$ (again, by Lemma \ref{lemma:crit}), and
finally the last equation  is valid by the definition of $I_{H_k}^*$. This completes the proof of the theorem.
\end{proof}

As a  nice application of Theorem~\ref{thm:main} one obtains  the result  that in the shifted family of a numerical semigroup certain
homological properties occur  for all large shifts simultaneously for the semigroup ring and its tangent cone.

\begin{Corollary}
\label{cor:persistence}
Let $H$ be a numerical semigroup. There exists a positive integer $k_0$ such that for any $k \geq k_0$
the ring $K[H_k]$  is  complete intersection, respectively Gorenstein, if and only if $\gr_\mm K[H_k]$ has this property.

Moreover, in the shifted family $\{H_k \}_{k\geq 0}$, the property of the tangent cone $\gr_\mm K[H_k]$ to be a
complete intersection, respectively Gorenstein, occurs eventually periodically.
\end{Corollary}

\begin{proof}
Pick $k_0$ as given by Theorem \ref{thm:main} applied to the semigroup $H$.

The first part of the corollary follows from the fact that $K[H_k]$ and  $\gr_\mm K[H_k]$
have the same codimension and the same Betti numbers.

The fact about periodicity arises  from the eventual periodicity of the Betti numbers for $K[H_k]$,  see Theorem \ref{thm:vu-periodicity}.
\end{proof}

We define the {\em width} of a numerical semigroup $H$ as the difference between
the largest and the smallest generator in a minimal set of generators of $H$, and denote this number by $\width (H)$.
Notice that any semigroup $H_k$ in the shifted family of $H$ has $\width (H_k) \leq  \width(H)$, with equality for $k\gg 0$.

\medskip
As an immediate consequence of   our Theorem \ref{thm:main} and of  Theorem \ref{thm:vu-periodicity}  we obtain

\begin{Corollary}
\label{cor:width}
Let  $w\geq 2$ and let $\mathcal{H}_w$ be the set of all numerical semigroups $H$ with $\width(H) \leq w$.
Then for any integer $i\geq 0$ there exists an integer $b$ such that
$$
\beta_i(I_{H}^*) \leq b \quad \text{for all} \quad H \in \mathcal{H}_w.
$$
\end{Corollary}

\begin{proof}

Let $\mathcal{A}$ be the set of all strictly increasing sequences of integers $\ab$ with first term $0$ and last term at most $w$.
Given a numerical semigroup $H$ with $\width(H) \leq w$, there exists a unique $\ab \in \mathcal{A}$ and a
unique integer $k$ such that $H= \langle \ab +k \rangle$.

Therefore, given $i$, it suffices to show that there exists $b$ such that
\begin{eqnarray}
\label{above}
\beta_i(I(\ab +k)^*) \leq b \quad \text{ for all $\ab \in \mathcal{A}$ and all $k \geq 0$}.
\end{eqnarray}
Since $\mathcal{A}$ is finite, we only need to show \eqref{above}   for any fixed
$\ab \in \mathcal{A}$ and all $k$.

Now fix $\ab \in \mathcal{A}$. By Theorem \ref{thm:main}, there exists an integer $k_0$ such that
for all $k \geq k_0$, $\beta_i(I(\ab+k)^*)= \beta_i(I(\ab+k))$. To conclude the proof, we use  Theorem \ref{thm:vu-periodicity} from
which it follows that there exists an integer $k_1 \geq k_0$ and an integer $b_1$
such that
$$
\beta_i(I(\ab+k)) \leq b_1 \quad \text{ for all $k \geq k_1$}.
$$
Let
$$
  b_0= \max\{ \beta_i(I(\ab +k)^*\: k\leq k_1\},
$$
and set $b= \max \{b_0, b_1\}$. Then
$\beta_i(I(\ab+k)^*) \leq b$ for all $k$, as desired.
\end{proof}

\section{Expected bounds for $\mu(I_H^*)$}
\label{sec:conjectures}

It would be nice to have an explicit value for the bound $b$ in Corollary \ref{cor:width}  in terms of the width of the semigroup.
Computer calculations with  CoCoA  \cite{Cocoa} and SINGULAR  \cite{Sing}   suggest  us to formulate the following conjecture.

\begin{Conjecture}
\label{conj:one} If $H$ is a numerical semigroup, then   $\mu (I^*_H) \leq {\width(H)+1 \choose {2}}$.
If $\mu(H) \geq 2$, then equality holds if and only if there exist integers  $w,k \geq 1$ such that
$$
H = \langle kw+1, kw+2, \dots,  (k+1)w+1 \rangle.
$$
\end{Conjecture}

Observe that this conjecture implies in particular that $\mu(I_H) \leq  {{\width(H)+1} \choose {2}}$.
We verified Conjecture \ref{conj:one}   for all numerical semigroups whose width is at most $5$.
We did this as follows: for a fixed width $w \leq 5$ we considered all sequences of strictly increasing  integers
$\ab =  a_1  < \dots <a_r$ with $a_1=0$ and $a_r=w$. For such a sequence we computed the values of $\mu(I(\ab+k)^*)$ when we let $k$ vary.
According to  Vu's Theorem \ref{thm:vu-periodicity} and our Theorem \ref{thm:main},
there exists an integer $k_{\ab}$ such that for all $k \geq k_{\ab}$ the values of $\mu(I(\ab+k)^*)$ become periodic with period $w$.
For each of our sequences $\ab$ we have identified the value of $k_{\ab}$ and by inspection of  $\mu(I(\ab+k)^*)$ for $k < k_{\ab}+w$ we  verified Conjecture \ref{conj:one}.

\medskip
Numerical experiments allow us to formulate an even stronger claim. Before we state it, let us give a couple of definitions.

Let $H$ be a numerical semigroup minimally generated by $a_1 < \dots <a_r$.  We let $\widetilde{H} = \langle a_1, a_1+1, a_1+2, \dots, a_r \rangle$ be the semigroup generated by
 all integers in the interval $[a_1,a_r]$. We call  $\widetilde{H}$ the {\em interval completion} of $H$.

If a numerical semigroup $H$ is generated by all the integers of an interval, we call it an {\em interval semigroup}.
Clearly, if $H$ is any numerical semigroup, its interval completion $\widetilde{H}$ is an interval semigroup.

With notation as above, the integers in the interval $[a_1, a_r]$ may not always be a minimal generating set for  $\widetilde{H}$.
For example, if $H= \langle 3, 7 \rangle$ then we have $\widetilde{H} = \langle 3,4,5,6,7 \rangle = \langle 3,4,5 \rangle$.

Let $i \leq \width(H)$ be a positive integer. Then $a_1+i$ is not a minimal generator for $\widetilde{H}$ if  $i \geq a_1$.
Hence we get

\begin{Lemma}
\label{lemma:Htilde}
For any numerical semigroup $H$ one has  $\widetilde{H} = \langle a_1 , \dots, a_1+ \width(\widetilde{H})\rangle$, where
$$
\width(\widetilde{H})  =\mu(\widetilde{H}) -1=  \min \{ a_1 -1 , \width(H)   \}.
$$
\end{Lemma}

As a consequence we obtain
\begin{Lemma}
\label{lemma:mu-tilde}
If $H$ is a numerical semigroup, then
$\mu(H) \leq \mu (\widetilde{H})$.
\end{Lemma}

\begin{proof}
With notation as above,  if $a_1 < \width(H) +1$, then $ a_1= \mu(\widetilde{H})$.
Suppose that $\mu(H) > a_1$. Then there exist two distinct minimal generators for $H$, say  $b$, $c$, such that $b \equiv c \mod a_1$, a contradiction.

If $a_1 \geq \width(H) +1$, then $\mu(\widetilde{H}) = \width(H) +1 $ and by Lemma \ref{lemma:Htilde} we have that $\widetilde{H}$ is minimally generated by the whole interval
$[a_1, a_r]$, which clearly includes the minimal generating set of $H$. Hence $\mu(H) \leq \mu(\widetilde{H})$ in this case, too.
\end{proof}


Observe that $H$ and $\widetilde{H}$ may have the same number of generators, although they are different.  For instance, if $H= \langle 3,5,7 \rangle$, then
$\widetilde{H}= \langle 3,4,5 \rangle$.
With notation as above, if $\width(H)  \leq  a_1-1$, then $\mu(H) = \mu(\widetilde{H})$ if and only if $\mu(H)= a_r-a_1+1$, equivalently $H=\widetilde{H}$.

If $\width(H) >a_1-1$, we have $\mu(H)= \mu(\widetilde{H})$ if and only if $r= a_1$, equivalently $\widetilde{H}=\langle r , \dots, 2r-1 \rangle$. For any fixed $r$ there
are usually several numerical semigroups $H$ minimally generated by $a_1=r <a_2 < \dots < a_r$ and such that $\widetilde{H}=\langle r , \dots, 2r-1 \rangle$.
A necessary condition for that to happen is that $a_i \not\equiv a_j \mod r$ for all $1\leq i < j \leq r$.

\medskip
We can now state

\begin{Conjecture}
\label{conj:two}
Let $H$ be a numerical semigroup. Then $\mu(I_H^*) \leq \mu(I_{\widetilde{H}}^*)$.
\end{Conjecture}

In the following we explain why a positive answer to Conjecture \ref{conj:two} will give a positive answer to the first part
of Conjecture \ref{conj:one}.
Indeed, if Conjecture \ref{conj:two} holds, we   may apply
Proposition \ref{prop:consecutive} together with Lemma \ref{lemma:Htilde}, and consequently   we obtain
$$
\mu(I_H^*) \leq \mu (I^*_{\widetilde{H}}) \leq {\width(\widetilde{H})+1 \choose 2 }  \leq {\width(H)+1 \choose 2 }.
$$
Hence the inequality in Conjecture \ref{conj:one} is valid, too.

\medskip
Next we will show that Conjecture~\ref{conj:two} holds true for a numerical semigroup which is generated by an arithmetic sequence.
Actually, we will  show  in Proposition \ref{prop:conj2-arithmetic} that in this case one even has $\beta_i(I_H^*) \leq \beta_i(I_{\widetilde{H}}^*)$ for all $i$.

Recall that a sequence of integers $a_1,  a_2, \dots, a_r$ with $r\geq 2$ is called an   {\em arithmetic sequence} if there exists  a positive  integer $d$
such that $a_i- a_{i-1}=d$ for all $i=2,\dots, r$.
The class of numerical semigroups $H$ generated by arithmetic sequences has received much attention due to the  extra structure.
A minimal system of generators for the ideal $I_H$ was presented by D.P.~Patil in \cite{Patil}. Recently, L.~Sharifan and  R.~Zaare-Nahandi have obtained
explicit formulas for the graded Betti numbers of $\gr_\mm K[H]$, see \cite[Theorem 4.1]{Sha-Za-1}. Independently, P.~Gimenez, I.~Sengupta and H.~Srinivasan
found the minimal free resolution and a formula for the Betti numbers of $K[H]$,  see  \cite{GSS}. By inspecting the two sets of formulas, it was noted in
\cite{Sha-Za-2} that $\beta_i (K[H]) = \beta_i (\gr_\mm K[H])$ for all $i$. By using Lemma \ref{lemma:crit}, we give a more conceptual proof
of this result in Proposition \ref{prop:arithmetic}.

The sequence of integers  $a_1,  a_2, \dots, a_r$ with $r\geq 2$ is called a   {\em generalized arithmetic sequence} if there exist integers $h$ and $d$  such that
$a_i = h a_1 + (i-1)d$ for $i=2, \dots, r$. By the work of L.~Sharifan and  R.~Zaare-Nahandi in \cite{Sha-Za-2},
one can show with  minor changes to our proofs that  all the results from  the rest of this section are   also valid for semigroups $H$
generated by generalized arithmetic sequences.
For simplicity, in what follows we only consider arithmetic sequences.

\medskip
We first present  our alternative proof  of the following proposition due to   L.~Sharifan and  R.~Zaare-Nahandi.

\begin{Proposition}
\label{prop:arithmetic}
Let $H $ be a numerical semigroup minimally generated by the arithmetic sequence $\ab= a_1 < a_2 < \dots < a_r$.
Then $\beta_i(I_H^*) = \beta_i (I_H) $ for all $i$.
\end{Proposition}
\begin{proof}

If $r <3$, the statement is immediate.

Assume $r\geq 3$.
Let $a_i= a_1 + (i-1)d$ \ for $i=1, \dots, r$ and $d$ a positive integer.
Without loss of generality, we may assume that $\gcd(a_1, d)=1$. Otherwise, we divide the sequence $\ab$ by $\gcd(a_1, d)$   and we obtain
an arithmetic sequence $\bb$ with the desired property.
The semigroups $\langle \ab \rangle$ and $\langle \bb \rangle$ are isomorphic and so are their associated semigroup rings.

We describe the minimal system of generators of the ideal $I_H$ following the presentation in \cite{GSS}.

Let $a$ and $b$ be  the unique positive
integers such that $a_1= a (r-1) +b$  with $1\leq b \leq r-1$. Consider the following matrices of variables:
$$
A= \left(
\begin{matrix}
x_1  & \cdots & x_{r-1} \\
x_2 &\cdots & x_{r}
\end{matrix}
\right),
\quad
B= \left(
\begin{matrix}
x_r^a & x_1 & \dots  & x_{r-b} \\
x_1^{a+d} & x_{b+1} & \dots & x_r
\end{matrix}
\right).
$$

Let $\Delta_i$ be the maximal minor of $B$ involving the first and the $(i+1)${st} column for $i=1,\dots, r-b$,
$$
\Delta_{i}= \left| \begin{matrix}  x_r^a &  x_i \\  x_1^{a+d} & x_{b+i} \end{matrix} \right|  = x_r^a x_{b+i} - x_1^{a+d}x_i \quad \text{for }i=1,\dots, r-b.
$$

Let $\xi_{ij}= \left|\begin{matrix} x_i & x_j \\ x_{i+1} & x_{j+1}\end{matrix} \right|= x_ix_{j+1}- x_{i+1}x_j$ for $1\leq i < j \leq r-1$ be the maximal minors of $A$.
It is known from   \cite{Patil} and \cite[Theorem 1.1]{GSS1} that  $I_H$ is minimally generated by   the $\xi_{ij}$'s and the $\Delta_i$'s:
\begin{eqnarray}
\label{eq:mingen}
I_H= (\xi_{ij}: 1\leq i < j \leq r-1) + (\Delta_1, \dots,\Delta_{r-b}).
\end{eqnarray}

We claim that these generators  also form a standard basis of $I_H$.

Let $S=K[x_1, \dots, x_r]$ and    consider the substitution homomorphism $\pi$ defined on $S$
by $\pi(x_1)=0$ and $\pi(x_i)=x_i$ for $i>1$,  as in  Lemma \ref{lemma:crit}.
Note that the ideal $\pi(I_H)$ is homogeneous and   its generators coming from \eqref{eq:mingen}
can be lifted via $\pi$ to polynomials in $I_H$  with the same initial degree. Therefore, applying Lemma \ref{lemma:crit} we conclude that the generators of $I_H$ given in
\eqref{eq:mingen} form a standard basis.

We also have that $x_1$ is a regular element on $\gr_\mm K[H]$ and arguing as in the proof of Theorem \ref{thm:main}
we conclude that $\beta_i (I_H^*)= \beta_i(I_H)$ for all $i$.
\end{proof}

We recall here the  formula given in \cite{GSS} for the Betti numbers of $K[H]$.
\begin{Theorem}\;{\em(Gimenez, Sengupta and  Srinivasan \cite[Theorem  4.1]{GSS}) }
\label{thm:betti-arithmetic}

Let $H$ be a numerical semigroup generated minimally by the arithmetic sequence $a_1 <\dots < a_r$ with $\gcd(a_1, a_2)=1$.
Let $b$ the unique integer such that $a_1 \equiv b \mod (r-1)$ and $1\leq b \leq r-1$. Then
\begin{eqnarray}
\label{eq:betti}
\beta_i(K[H])= i { r-1 \choose i+1} +
\begin{cases}
(r- b-i +1) {r-1 \choose  i-1}  &\text{ if }  1 \leq i \leq r-b, \\
(i-r+ b) {r-1 \choose i} & \text{ if } r-b < i \leq r-1.
\end{cases}
\end{eqnarray}
\end{Theorem}

\medskip
Note that from the Auslander-Buchsbaum Theorem it follows that $\projdim K[H]=r-1$, therefore \eqref{eq:betti} displays
all non-zero Betti numbers $\beta_i(K[H])$ for $i>0$.

It is surprising that according to \eqref{eq:betti}, the Betti numbers of  a semigroup ring $K[H]$ associated to an arithmetic sequence
$a_1 <\dots <a_r$ do not depend on $d=a_2-a_1$, but only on the number of minimal generators  $r=\mu(H)$ and the residue $a_1 \mod(r-1)$.


By using Theorem \ref{thm:betti-arithmetic} we obtain the following result which will be crucial for our further consideration.
\begin{Proposition}
\label{prop:betti-bounds}
Let $H$ be a numerical semigroup minimally generated by the arithmetic sequence  $a_1 <\dots < a_r$.
Then
\begin{eqnarray}
\label{eq:betti-bounds}
i {r-1 \choose i+1} < \beta_i(K[H]) \leq i {r \choose i+1} \quad \text{ for all \ } 1\leq i \leq r-1.
\end{eqnarray}
Moreover, if we let $e=\gcd(a_1, a_2)$, the following statements are equivalent:
\begin{enumerate}
\item[(i)] $\beta_i(K[H]) = i {r \choose i+1}$ for some $i$ with $1\leq i \leq r-1$,
\item[(ii)] $\beta_i(K[H]) = i {r \choose i+1}$ for all $i$ with $1\leq i \leq r-1$,
\item[(iii)]  $a_1 \equiv e \mod e(r-1)$.
\end{enumerate}
\end{Proposition}

\begin{proof}
We first consider the case $e=1$. Let $b$ the unique integer such that $a_1 \equiv b \ \mod (r-1)$ and $1\leq b \leq r-1$.
The first inequality in \eqref{eq:betti-bounds} is trivial, as the second summand in the Betti-formula \eqref{eq:betti} is in either case positive.
By using Pascal's formula
\begin{eqnarray*}
{r \choose i+i} = {r-1 \choose i+1} + {r-1 \choose i}
\end{eqnarray*}
and the identity $(r-i) {r-1 \choose  i-1} = i {r-1 \choose i}$, we may rewrite \eqref{eq:betti} as follows:
\begin{eqnarray}
\label{eq:betti-new}
\beta_i(K[H])= i { r\choose i+1} -
\begin{cases}
(b-1) {r-1 \choose  i-1}  &\text{ if }  1 \leq i \leq r-b, \\
(r-b) {r-1 \choose i} & \text{ if } r-b < i \leq r-1.
\end{cases}
\end{eqnarray}

By our choice of $b$ we have that $b-1 \geq 0$ and $r-b >0$. This leads immediately to the inequality
\begin{eqnarray*}
\label{eq:something}
\beta_i(K[H]) \leq i {r \choose i+1}.
\end{eqnarray*}
If for some $i$ with $1\leq i \leq r-1$  one has that $\beta_i(K[H]) = i {r \choose i+1}$, then $b=1$.
Therefore,  the second case in the Betti-formula \eqref{eq:betti-new} does not apply  and $\beta_i(K[H]) = i {r \choose i+1}$  for all $i>0$.

If $e>1$, then we let $H'$ be the semigroup obtained from $H$ by dividing all generators of $H$ by $e$.
Since $\beta_i(K[H']) =\beta_i(K[H])$ for all $i$, the inequalities \eqref{eq:betti-bounds} follow from the case $e=1$.
Moreover, the desired equivalences  follow from the first part of this proof
and the observation that  $a_1 \equiv e \mod e(r-1)$ if and only if $b_1 \equiv  1 \mod (r-1)$, where we let $b_1= a_1/e$.
\end{proof}

Our next result shows that a more general form of Conjecture \ref{conj:one} is true for semigroups generated by an arithmetic sequence.

\begin{Proposition}
\label{prop:consecutive}
Let $H$ be a numerical semigroup generated by an arithmetic sequence. Then
\begin{eqnarray}
\label{eq:betti-width}
\beta_i(\gr_\mm K[H]) \leq i { \width(H) + 1 \choose i+1} \text{ for all }  i >0.
\end{eqnarray}
Equality holds for some $i$ with $1\leq i < \mu(H)$ if and only if there exist integers  $w,k \geq 1$ such that
$$
H= \langle kw+1, kw+2, \dots,  (k+1)w+1 \rangle \quad \text{with} \quad w, k \geq 1.
$$
In this case \eqref{eq:betti-width} becomes an equality for all $i$ with $1\leq i <\mu(H)$.
\end{Proposition}

\begin{proof}
Let $\ab = a_1 < a_2< \dots <a_r$ be the arithmetic sequence that minimally generates $H$.
By using Proposition \ref{prop:arithmetic} and Proposition \ref{prop:betti-bounds} it follows that
$$
\beta_i(\gr_\mm K[H]) \leq i{ r \choose i+1} \leq i { \width(H) +1 \choose i+1}.
$$
The second inequality becomes equality (independently of $i$) if and only if $a_2-a_1=1$.
Thus, by using Proposition \ref{prop:betti-bounds} it follows that $\beta_i(\gr_\mm K[H]) = i { \width(H) +1 \choose i+1} $ if and only if  $a_1 \equiv 1 \mod (r-1)$.

If we let $w=\width(H)$, then there exists a positive integer $k$ with $a_1=kw+1$. Hence $a_r= a_1+ (r-1)d= kw+1+w$.
This completes the proof.
\end{proof}

\medskip
The following statement shows that the  Betti numbers of semigroup rings associated to arithmetic sequences increase
with the number of terms in the sequence.

\begin{Proposition}
\label{prop:betti-easy-prop}
Let $H$ and $H'$ be numerical semigroups generated by arithmetic sequences such that $\mu(H) < \mu(H')$.
Then  $\beta_i(K[H]) < \beta_i(K[H'])$ for all $i$ such that $0<i  < \mu(H')$.
\end{Proposition}

\begin{proof}

By using Proposition \ref{prop:betti-bounds} we have that
$$
\beta_i(K[H]) \leq i {\mu(H) \choose i+1} \leq i {\mu(H') -1 \choose i+1} < \beta_i (K[H']).
$$
This implies the desired conclusion.
\end{proof}

The final result in this section shows that a stronger version of Conjecture \ref{conj:two} is valid for semigroups generated by an arithmetic sequence.
This stronger version may be even true for any numerical semigroup.

\begin{Proposition}
\label{prop:conj2-arithmetic}
Let $H$ be a numerical semigroup generated by an arithmetic sequence.
Then $\beta_i(I_H^*) \leq  \beta_i(I_{\widetilde{H}}^*)$ for all $i$.  The equality is achieved for all $i$ if and only if
$H$ is generated by consecutive positive integers, or if $H=\langle r, r+d, \dots, r+ (r-1)d \rangle$ for some
positive integers $r$ and $d$ with $\gcd(r,d)=1$ and $r>2$.
 \end{Proposition}

\begin{proof}
By Proposition \ref{prop:arithmetic} it suffices to show that
\begin{eqnarray*}
\beta_i({I_H})\leq \beta_i({I_{\widetilde{H}}}) \text{ for all } i.
\end{eqnarray*}
These inequalities are  an immediate consequence of Lemma \ref{lemma:mu-tilde} and Proposition~\ref{prop:betti-easy-prop}.

Any of these inequalities turn into  an equality if and only if $\mu(H) = \mu(\widetilde{H})$.
 This is obviously true when  $H$ is (minimally) generated by some consecutive numbers, i.e. $H=\widetilde{H}$.
If that is not the case, by the discussion before Conjecture \ref{conj:two} it follows that
$$
H= \langle r, r+d, \dots, r+ (r-1)d \rangle
$$
with $\mu(H)=r$ and such that its $r$ minimal generators give different remainders $\mod r$, equivalently
that $\gcd(r,d)=1$.

We claim that for any $r>2$ and $d>1$ such that $\gcd(r,d)=1$ the numbers $r, r+d, \dots, r+ (r-1)d$ are a minimal
generating set for the semigroup they span. Indeed, suppose there exists an integer $i$ with $0<i \leq r-1$
and integers $a_j\geq 0$ such that
$$
r+id= \sum_{j=0}^{i-1}a_j(r+jd).
$$
It follows that
$$
(\sum_{j=0}^{i-1}a_j-1)r= (i- \sum_{j=0}^{i-1}j a_j)d.
$$
Since at least one $a_j>0$, it follows that 
$\sum_{j=0}^{i-1}a_j-1 \geq 0$. This implies that $r> i\geq i-\sum_{j=0}^{i-1}j a_j \geq 0$.

On the other hand, since $\gcd(r,d)=1$, it follows that $r$ divides $i- \sum_{j=0}^{i-1}j a_j$. This is only possible if
$i-\sum_{j=0}^{i-1} j a_j =0$, in which case $\sum_{j=0}^{i-1}a_j-1 = 0$.
This implies that precisely only one of the $a_j=1$ while the others are zero. 
Thus $i-j= i- \sum_{j=0}^{i-1}j a_j =0$, a contradiction.
 
\end{proof}


\section{Examples}
\label{sec:examples}

In this section we study the defining ideals of the tangent cone of semigroup rings for  several families of numerical semigroups
and compare their number of generators with  our conjectured upper bound in   Conjecture~\ref{conj:one}.
The first two families are due to J.~Sally \cite{Sally} and H.~Bresinsky \cite{Bres}.  The last two families are families of three generated semigroups,
one of which has been considered by T.~Shibuta in \cite{Goto}.

\medskip
\noindent
\subsection{\em Sally semigroups.} In \cite{Sally} Judith Sally considered  Gorenstein local rings whose multiplicity is small compared to the embedding dimension of the ring and gave explicitly a minimal set of generators of the defining ideal of the tangent cone.

To be precise, let $(R,\mm)$ be a Gorenstein local ring of dimension $d$, embedding dimension $r$ and multiplicity $e>4$ such that $r=e+d-3$. Her Theorem 3 says the following:
assume  further that $R/\mm$  is infinite and that $R$ has a
presentation, $R = S/I$, where $S$ is a regular local ring of dimension $e + d- 3$. Then
\[
\gr_\mm R \iso R/\mm[x_1,\ldots, x_d, y, z_1,\ldots, z_{e - 4}]/I^*,
\]
where $I^*$  is minimally generated by the ${e-2\choose 2}$ monomials
\[
yz_1,\ldots,yz_{e-4},\ z_iz_j,  1\leq i\leq j\leq e-4, \text{ and $y^4$}.
\]

There are plenty of numerical semigroups whose complete semigroup ring is a Gorenstein ring with $r=e+d-3 =e-2$. A numerical semigroup with this property will be called a  {\em Sally semigroup}.

It would be interesting to find all Sally semigroups. For any given $e\geq 4$  we give  an example of a Sally semigroup $S_e$ of multiplicity $e$. Let
\[
S_e= \langle i \: e \leq i \leq 2e-1, i\neq e+2, e+3\rangle.
\]
Obviously, $\mu(S_e)=e-2$. Moreover, $S_e$ is symmetric because its  Frobenius number is equal to $2e+3$  and
\[
S_e=\{0,e,e+1, e+4,\ldots, 2e+2, 2e+4, 2e+5,\ldots\}.
\]
Hence,  by a theorem of Kunz \cite{Kunz},  it follows that $K[[S_e]]$ is Gorenstein.

By the above mentioned theorem of Sally we have $\mu(I_{S_e}^*)={e-2\choose 2}$. Our conjectured upper  bound in this particular case is ${e\choose 2}$.

More generally, if $H$ is any Sally semigroup of multiplicity $e$, then   it verifies Conjecture \ref{conj:one}:
\[
\mu(I_H^*)= {e-2\choose 2}={r\choose 2}\leq {\width(H)+1\choose 2}.
\]

\medskip
\noindent
\subsection{\em Bresinsky semigroups.} In 1975\ H.\ Bresinsky \cite{Bres} introduced the following family   of $4$-generated numerical semigroups.
Given an  integer $h\geq 2$ we let
 $$
B_h=\langle (2h-1)2h, (2h-1)(2h+1), 2h(2h+1), 2h(2h+1)+2h-1\rangle.
$$
We claim that  $\mu(I_{B_h}^*)=\mu(I_{B_h})= 4h$ and that $\gr_\mm K[B_h]$ is Cohen-Macaulay.

\medskip
Fix $h\geq 2$ and let $I=I_{B_h} \subset K[x,y,z,t]$ be the relation ideal of $K[B_h]$.
We will give  a minimal standard basis of $I$ with $4h$ elements. It is proved in \cite[Lemma 3]{Bres} that
$I  = (\mathcal{A}_1 \cup \{g_1= z^{2h-1}-y^{2h}, g_2= xt-yz\} \cup \mathcal{A}_2)$, where
$$ \mathcal{A}_1 = \{ f_i = z^{i-1}t^{2h-i}- y^{2h-i}x^{i+1} \:\ 1\leq i \leq 2h\} $$
and
$$
\mathcal{A}_2 = \{f=x^{\nu_1}z^{\nu_3}-y^{\mu_2}t^{\mu_4} \: \  \nu_3, \mu_4 < 2h-1, f\in \ I \}.
$$

The set $\mathcal{A}_2$ is infinite, so we consider a finite subset, namely
$$
\mathcal{A}_3=   \{f=x^{\nu_1}z^{\nu_3}-y^{\mu_2}t^{\mu_4} \: f\in \mathcal{A}_2, \mu_2 \leq 2h \}.
$$
We claim that $\mathcal{A}_2$ and $\mathcal{A}_3$ generate the same ideal in $K[x,y,z,t]$.
Pick $f=x^{\nu_1}z^{\nu_3}-y^{\mu_2}t^{\mu_4}$ in $\mathcal{A}_2$ with $\mu_2 >2h$.  We show that $f \in   (\mathcal{A}_3)$.
Indeed, write $\mu_2= 2h\cdot \alpha + \beta$, with $\alpha$, $\beta$ integers and
$0\leq \beta < 2h$. We may rewrite
\begin{eqnarray*}
f &=& x^{\nu_1}z^{\nu_3}-y^{\mu_2}t^{\mu_4} = x^{\nu_1}z^{\nu_3}-y^{ 2h \cdot \alpha + \beta}t^{\mu_4} \\
 &=& x^{\nu_1}z^{\nu_3}-y^{\beta}t^{\mu_4}((y^{2h})^\alpha - (x^{2h+1})^\alpha)- x^{(2h+1)\alpha}y^{\beta}t^{\mu_4}.
\end{eqnarray*}
Since $(y^{2h})^\alpha - (x^{2h+1})^\alpha \in (\mathcal{A}_3)$, it suffices to show that $g= x^{\nu_1}z^{\nu_3} - x^{(2h+1)\alpha}y^{\beta}t^{\mu_4} \in (\mathcal{A}_3) $.
We may assume $g\neq 0$.
 If $\nu_1 \leq (2h+1) \alpha$, since $I$ is a prime binomial ideal, we should also have
$0\neq z^{\nu_3}- x^{\alpha'} y^{\beta}t^{\mu_4} \in I$ for some nonnegative integer $\alpha'$.
 This is not possible because $\nu_3 <2h-1$ and by \cite[Lemma 1]{Bres}, we must have $\nu_3 \geq 2h-1$.
Therefore $\nu_1 > (2h+1) \alpha$ and  we may write $g= x^{(2h+1)\alpha}( x^{\gamma}z^{\nu_3}-y^{\beta}t^{\mu_4}) $ with $\gamma$ a nonnegative integer.
It follows that $g\in (\mathcal{A}_3)$, as desired.

As a consequence of the above $I=(\mathcal{A}_1, g_1, g_2, \mathcal{A}_3)$.
Next we  will find explicitly the polynomials in $\mathcal{A}_3$.
Let $f=x^{\nu_1}z^{\nu_3}-y^{\mu_2}t^{\mu_4}  \in \mathcal{A}_3$.
Then
\begin{eqnarray*}
\label{eq:bres}
 \nu_1 2h(2h-1) + \nu_3 2h(2h+1) = \mu_2 (2h-1)(2h+1)+ \mu_4 (2h(2h+1) +2h-1).
\end{eqnarray*}
Reducing this equation modulo $2h$, we see  that $2h$ divides $\mu_2 + \mu_4$. Since $\mu_4 < 2h-1$ and $\mu_2 \leq 2h$, we obtain that $\mu_2 + \mu_4 = 2h$.
By using this fact and dividing both sides of the above equation  by $2h$ we get
\begin{eqnarray*}
\nu_1 (2h-1) + \nu_3 (2h+1) = \mu_2 \cdot 2h+ \mu_4(2h+2)-1= 4h^2-1+2\mu_4.
\end{eqnarray*}
It follows that $ \nu_3\equiv \mu_4 \mod (2h-1)$. Since $\nu_3, \mu_4 < 2h-1$, we obtain that $\nu_3 = \mu_4$.
This implies that $\nu_1+ \nu_3 = 2h+1$. Thus,
$$
\mathcal{A}_3= \{ u_j= x^{2h+1-j}z^j-y^{2h-j}t^j \:\ 0 \leq j \leq 2h-2 \}.
$$
Note that $f_{2h} + u_0= z^{2h-1}-y^{2h}=g_1$.
We claim that the following set of $4h$ elements
$$
\mathcal{B}= \{f_i \: \ 1\leq i \leq 2h\} \cup \{g_2 \} \cup \{ u_j \: \ 0\leq j \leq 2h-2 \}
$$
is a minimal generating set and a minimal standard basis of $I$.

Indeed, with notation as in Lemma \ref{lemma:crit}, we observe that
$$
\bar{I}= ( t^{2h-1}, zt^{2h-2}, \dots, z^{2h-1})+(yz) + (y^{2h},  y^{2h-1}t, \dots,  y^3t^{2h-3}, y^2t^{2h-2}),
$$
which happens to be a  monomial ideal. Since $x$ is regular on $K[B_h]$ we  have that $\mu(I)= \mu(\bar{I})= 4h$.
The monomial generators of   $\bar{I}$, which form a standard basis, admit  the lifting property required in Lemma \ref{lemma:crit}.
Therefore  $\mathcal{B}$ is a standard basis of $I$ and
$$
I^*= (t^{2h-1}, zt^{2h-2}, \dots, z^{2h-1})+ (xt-yz) + (y^{2h},  y^{2h-1}t, \dots,  y^3t^{2h-3}, y^2t^{2h-2}).
$$

As a consequence of Lemma \ref{lemma:crit} we have that the tangent cone $\gr_\mm K[B_h]$ is Cohen-Macaulay.
This was also announced by F.\ Arslan in \cite[Remark 3.8 (b)]{Arslan}.

\bigskip
Before considering families of $3$-generated semigroups, we  first recall how one determines the relation ideal of such semigroups.
Let $H=\langle n_1, n_2, n_3 \rangle$,
where $n_1 <n_2 <n_3$ are the minimal generators. According to the original paper \cite{He-semi},
the defining ideal $I_H$ of the semigroup algebra $K[H]\cong K[x,y,z]/I_H $ is generated by $2$ or $3$ elements which can
be easily found  from $n_1, n_2, n_3$.
For each $n_i$, $i=1,\dots, 3$ one takes the least positive multiple $c_i n_i$ that lies in the semigroup generated by the other two generators, and obtains
\begin{eqnarray}
\label{eq:3gen-additive}
\left\{
\begin{matrix}
c_1 n_1 = r_{12} n_2 + r_{13}n_3, \\ 
c_2 n_2 = r_{21} n_1 + r_{23}n_3, \\
c_3 n_3= r_{31} n_1 + r_{32}n_2.  
\end{matrix}
\right.
\end{eqnarray}

Given this data, the ideal $I$ is   generated by
\begin{eqnarray}
\label{eq:3gen}
f_1=x^{c_1}-y^{r_{12}}z^{r_{13}}, f_2=y^{c_2}-x^{r_{21}}z^{r_{23}}, f_3=z^{c_3}-x^{r_{31}}y^{r_{32}}.
\end{eqnarray}

Note that some $r_{ij}$ may be zero. In this situation two of the above polynomials are the same up to a sign
and the other coefficients $r_{st}$ are not necessarily unique.
If all $r_{ij} >0$, then all coefficients  are unique and
\begin{eqnarray}
\label{eq:sum}
c_1= r_{21}+r_{31}, \quad
c_2= r_{12} + r_{32}, \quad
c_3= r_{13}+ r_{23}.
\end{eqnarray}

\medskip
\noindent
\subsection{\em Variations on Shibuta semigroups.}
Let $a>3$ be an integer  and let $H_a= \langle a, a+1, 2a+3 \rangle$. (For $a=3k$ with  $k= 2,3\ldots$, this is the Shibuta family.)
We claim that  $\mu(I^*_{H_a}) = \lfloor \frac{a-1}{3} \rfloor +3$.

In this family  $\mu(I^*_{H_a})$ is a quasi-linear function of $a$ while our conjectured upper bound is a polynomial of degree $2$ in $a$.

\medskip
We indicate a proof of this claim.  To simplify notation, we  set $H=H_a$. We will describe a minimal standard basis of $I_{H}$.
The result and the proof  depend on $a \mod 3$.

Let $a=3k+1$, with $k>1$. Then $H= \langle 3k+1, 3k+2, 6k+5 \rangle$.
Any two of its generators are coprime, hence by \cite{He-semi}, $K[H]$ is not a complete intersection, and so
  all $r_{ij}$'s  in \eqref{eq:3gen-additive} are positive and unique.
First we prove that $I_H= (f_0, g, p)$, where
$$
f_0=yz^k-x^{2k+2}, \quad g=xz-y^3, \quad p=z^{k+1}-x^{2k+1}y^2.
$$
Clearly $f_0, g, p \in I_H$.  Therefore $c_2 \leq 3$. With notation as in \eqref{eq:3gen-additive}, if $c_2=2$, then the equation
$2(3k+2)= \alpha (3k+1) + \beta (6k+5)$ must have  a solution  with positive $\alpha$ and $\beta$, which is impossible.
Therefore, $c_2=3$ and $r_{21}=r_{23}=1$.
We reduce modulo $3k+1$ the  equation
$$
c_3 \cdot (6k+5)= r_{31} \cdot (3k+1) + r_{32} \cdot (3k+2)
$$
and we obtain
$$
3 c_3 \equiv r_{32} \ \mod (3k+1), \quad \text{where } c_3 \leq k+1.
$$
From \eqref{eq:sum} we derive that $0< r_{32} <3$. Hence the only possibility is to have $c_3=k+1$ and
$r_{32}=2$, $r_{31}=2k+1$. It follows by \eqref{eq:sum} and \eqref{eq:3gen} that $I_H=(f_0, g, p)$.

Next we build a minimal standard basis of $I_H$.
We recursively define a family of polynomials in $I_H$ by the rule:
$$
f_i = x f_{i-1} - y^{3i-1}z^{k-i}g, \quad \text{for  $i=1, \dots, k$}.
$$
One checks by induction on $i$ that $f_i= y^{3i+1}z^{k-i}-x^{2k+i+2}$ with  $f_i^{*}=y^{3i+1}z^{k-i}$ for  $i=0, \dots, k$ .

We show that $\mathcal{B} =\{g, p, f_0, f_1, \dots, f_k \}$ is a standard basis of $I_H$.
In order to prove this, we will use the homogenization technique described in \cite[\S15.10.3]{Eis}.
Namely, starting with the generating set $\mathcal{B}$ for $I_H$ we consider
the ideal $J$ in $K[s,x,y,z]$ generated by the homogenizations $q^h$
of the elements   $q\in \mathcal B$.
Next we find a Gr\"obner basis $\mathcal{G}$ of $J$ with respect to the lexicographic order
induced  by $s>x>y>z$. Then we dehomogenize the elements in
$\mathcal{G}$ and their initial forms will generate (not necessarily in a minimal way) the ideal $I_{H}^*$.

Let $f_{k+1}= xz f_{k}^h - y^{3k+1}g^h = y^{3k+4}-x^{3k+3}z \in J$. We claim that
$$
\mathcal{G} =\{ g^h, p^h,  f_{0}^h, f_{1}^h, \dots, f_{k}^h, f_{k+1} \}
$$
is a Gr\"obner basis for $J$ with respect to   the lexicographical order induced by $s>x>y>z$.
It is routine to check that all $S$-polynomials of pairs of elements in $\mathcal{G}$ reduce to $0$ with respect to $\mathcal{G}$.
Hence by Buchberger's Criterion (see for example \cite[Theorem 2.14]{EH} or \cite[Theorem 15.8] {Eis}) we conclude that $\mathcal{G}$ is a Gr\"obner basis of $J$.

We apply the algorithm described before and after we dehomogenize the elements in $\mathcal{G}$ we see that
$$
 I_H^*= ( xz, z^{k+1}, yz^k, y^4z^{k-1}, \dots, y^{3k+1})
$$
and $\mu(I^*_H)= k+3$.
(Note that $f_{k+1}^*= y^{3k+4}-x^{3k+3}z$  is already in the ideal generated by the other initial forms.)

\medskip
For the other two cases $a \equiv 0, 2 \mod 3 $ there is a similar discussion.
If $a=3k+2$ with $k>1$, then $H= \langle 3k+2, 3k+3, 6k+7 \rangle$, and
 $I_H= (f_0, g, p)$, where
$$
f_0=y^2z^k-x^{3k+3}, \quad g=xz-y^3, \quad p=z^{k+1}-x^{2k+1}y^2.
$$
We introduce recursively $f_i= x f_{i-1} - y^{2+3(i-1)}z^{k-i}g$, for $i=1, \dots, k$. Then by induction  on $i$ one shows that
$f_i= y^{3i+2}z^{k-i}- x^{2k+i+3} \in I_H$ and $f_{i}^*= y^{3i+2}z^{k-i}$ for $i=0, \dots, k$.
Then $\mathcal{C}=\{ g, p, f_0, \dots, f_k\}$ is a minimal standard basis of $I_H$ and
$$
I_H^*=(xz, z^{k+1}, y^2z^k, y^5z^{k-1}, \dots, y^{3k-1}z,y^{3k+2}).
$$
To see this, we homogenize the polynomials in $\mathcal{C}$, and check that together with $f_{k+1}= xzf_k-y^{3k+2}g$
they form a Gr\"obner basis with respect to the lexicographic order induced by $s>x>y>z$ for the ideal they generate.

\medskip
The last case   to consider is when $a=3k$, with $k\geq 2$. Then $H=\langle 3k, 3k+1, 3k+3 \rangle$.
This case is Shibuta's case in \cite[Example 5.5]{Goto}, where it was treated with tools different from the ones presented here.
For completeness we state here the relevant details. The semigroup ring $K[H]$ is a complete intersection and  $I_H= (f_0, g)$, where we let $f_0=z^k-x^{2k+1}$ and $g= xz-y^3$.
For $i=1,\dots, k$ we set  $f_i= x f_{i-1}- y^{3(i-1)}z^{k-i}g$. Then
$f_i = y^{3i}z^{k-i}-x^{2k+i+1} \in I_H$ and $f_i^*= y^{3i}z^{k-i}$ for $i=0, \dots, k$.
Finally one shows  that $\mathcal{D}= \{ g, f_0, f_1, \dots, f_k\}$ is a minimal standard basis of $I_H$, so that
$$
I_H^*= ( xz, z^k, y^3z^{k-1}, \dots, y^{3k} ).
$$
To show this, the  same technique as before works here. We homogenize the elements in $\mathcal{D}$, and together with
the polynomial $f_{k+1}=xzf_k -y^{3k}g$, these $k+3$ polynomials form  a Gr\"obner basis for the ideal they generate.
After dehomogenization we may discard $f_{k+1}$ from the standard basis since it has no contribution to $I^*_{H}$.
In all cases $\mu ( I_{H_a}^*)= \lfloor \frac{a-1}{3} \rfloor +3 $. This completes the proof.

\medskip
\noindent
\subsection{\em Frobenius semigroups.} Let $a, b >3 $ be coprime integers, and  consider the semigroup $H_{a,b}= \langle a, b, ab-a-b \rangle$. We call it a Frobenius semigroup because it is obtained from  the symmetric semigroup $\langle a,b\rangle$ by adding the last gap number of $\langle a,b\rangle$, called the Frobenius number, as an additional generator
to   $\langle a,b \rangle$  to obtain $H_{a,b}$. We claim that $\mu (I^*_{H_{a,b}}) =4$.

\medskip
This result is in so  far remarkable and untypical compared with the previous examples as the width of $H_{a,b}$ may be as large as we wish, while $\mu (I^*_{H_{a,b}})$ is always equal to $4$.

\medskip
In order to prove the claim we may assume without loss of generality  $a<b$. To simplify notation, we let $H= H_{a,b}$.
It is well known that $ab-a-b$ is the Frobenius number of the semigroup $\langle a, b\rangle$. This means that
$ab-a-b \notin \langle a, b\rangle$ and that if $s$ is an integer such that $s>ab-a-b$,  then $s\in H$, see \cite{Ramirez}.
Therefore $a <b <ab-a-b$ minimally generate $H$.

Let $I_H \subset S=K[x,y,z]$ be the relation ideal of the semigroup ring $K[H]$. Let
$$
f_1=x^{b-1}-yz, \quad f_2=y^{a-1}-xz, \quad f_3=z^2-x^{b-2}y^{a-2}.
$$
Clearly $f_1,f_2,f_3 \in  I_H$. We claim that these $f_i$'s correspond to  the minimal generators described in \eqref{eq:3gen}.
Indeed, as any two minimal generators of $H$ are coprime, it follows that $K[H]$ is not a complete intersection (see \cite{He-semi}).
Hence all the $r_{ij}$'s in \eqref{eq:3gen-additive} are positive and unique.  We have $c_3=2, r_{31}= b-2$ and $r_{32}=a-2$.
It follows from \eqref{eq:sum} that $c_1 > b-2$ and $c_2 > a-2$. Hence $f_1, f_2, f_3$ minimally generate $I_H$.

Let $f_4= x f_1 -y f_2 =x^b- y^a \in I_H$.  We claim that $\{f_1, f_2, f_3, f_4\}$ is a standard basis of $I_H$.
Let $J=(f_{1}^*, f_{2}^*, f_{3}^*,f_{4}^*)= (yz, xz, z^2, y^a) \subset I_{H}^*$.
By comparing their Hilbert series we show that $J=I_{H}^*$. This will then prove the claim.

For a finitely generated graded $S$-module $M=\bigoplus_{i \geq 0}M_i$ we consider its Hilbert series
$\Hilb_M(t)= \sum_{i\geq 0} \dim_K M_i \ t^i$.

Since $ J:z = (x,y,z)$, we obtain the exact sequence
\begin{eqnarray*}
0  \rightarrow ((x,y,z)/J) (-1) \rightarrow S/J(-1) \stackrel{z}{\rightarrow}  S/J \rightarrow K[x,y]/(y^a) \rightarrow 0.
\end{eqnarray*}
By using the additivity of the Hilbert series on exact sequences we get that
$$
\Hilb_{S/J}(t)= t+ \Hilb_{K[x,y]/(y^a)}(t).
$$

Let $\mm=(t^a, t^b, t^{ab-a-b}) \subset K[H]$ and  $\mm_1=(t^a, t^b) \subset K[\langle a, b\rangle]$.

Let $i \geq 2$. We show that for any positive integer $s$,
\begin{eqnarray}
\label{eq:iff}
t^s \in \mm^i \iff t^s \in \mm_{1}^i.
\end{eqnarray}

The implication ``$\Leftarrow$" is always satisfied.

For the converse, pick $t^s \in \mm^i$. Then  $s\neq ab-a-b$, since $i\geq 2$.
If $s< ab-a-b$, then the generator $t^{ab-a-b}$ does not appear in any presentation of $t^s$ as a product
of the monomial generators of $\mm$. Therefore, $t^s \in \mm_{1}^i$.

Finally suppose that $s > ab-a-b $. Then the desired implication follows from the fact that  for any decomposition
\begin{eqnarray*}
\quad s=\alpha a + \beta b + \gamma (ab-a-b)  \text{ with $\alpha, \beta, \gamma \in \ZZ_{+}, \ \alpha+\beta+\gamma \geq 2$ and $\gamma \geq 1$},
\end{eqnarray*}
there exist   $\alpha',\beta'\in \ZZ_{+}$ such that
\begin{eqnarray*}
s=\alpha a + \beta b + \gamma (ab-a-b)= \alpha' a + \beta' b   \ \text{ and } \alpha+\beta+\gamma < \alpha' + \beta'.
\end{eqnarray*}

Indeed, this can be deduced from the following identities:
$$
\begin{aligned}
 a+(ab-a-b) &= (a-1) b, \\
 b + (ab-a-b)&= (b-1)  a, \\
 2(ab-a-b)  &= (b-1) a + (a-2)\cdot b.
\end{aligned}
$$

A consequence of \eqref{eq:iff} is that the Hilbert series of the  tangent cones of  $K[H]$,
respectively of $K[\langle a, b \rangle]$ are the same  except in degree $i=1$, hence
\begin{eqnarray*}
\Hilb_{\gr_\mm K[H]} (t) &=& t+ \Hilb_{\gr_{\mm_1} K[\langle a,b \rangle]} (t) = t+ \Hilb_{ K[x,y]/(x^b-y^a)^*}(t) \\
 &=& t+ \Hilb_{K[x,y]/(y^a)} (t) =  \Hilb_{S/J}(t).
\end{eqnarray*}
Therefore $I_{H}^*=J$ and $\mu(I_{H}^*)=4$.

\medskip
{\bf Acknowledgement}.
The use of CoCoA \cite{Cocoa} and SINGULAR \cite{Sing} was vital for the development of this paper. We wish to thank their respective teams of developers.
We thank Mihai Cipu and the anonymous referee for reading the manuscript with a careful eye and for their suggestions.

The second author was  supported by  a grant of the
Romanian Ministry of Education, CNCS--UEFISCDI, project number PN-II-RU-PD-2012-3--0656.
He wishes to thank the  Fachbereich Mathematik, Universit\"at   Duisburg-Essen for the warm environment that fostered the collaboration of the authors.

{}

\end{document}